\documentclass[11pt]{article}

\usepackage[T1]{fontenc}
\usepackage[utf8]{inputenc}
\usepackage{lmodern}
\usepackage[a4paper,margin=1in]{geometry}

\usepackage{amsmath,amssymb,amsthm,mathtools}
\usepackage{enumitem}
\usepackage{graphicx}
\usepackage{booktabs}
\usepackage{hyperref}
\usepackage[nameinlink]{cleveref}
\usepackage{subcaption}
\usepackage{tikz}

\usepackage{listings}
\usepackage{xcolor}

\hypersetup{
  colorlinks=true,
  linkcolor=blue,
  citecolor=blue,
  urlcolor=blue,
  pdftitle={Trihexagonal Magic Figures: A High-Constraint-Density Triangular Combinatorial Design and the Possibility of a Phase Transition},
  pdfauthor={Donghwi Park}
}

\theoremstyle{plain}
\newtheorem{theorem}{Theorem}[section]
\newtheorem{proposition}[theorem]{Proposition}
\newtheorem{lemma}[theorem]{Lemma}

\newtheorem{conjecture}[theorem]{Conjecture}

\theoremstyle{definition}

\newtheorem{remark}[theorem]{Remark}

\title{\bfseries Trihexagonal Magic Figures: A High-Constraint-Density Triangular Combinatorial Design and the Possibility of a Phase Transition}

\author{
  Donghwi Park\thanks{}\\
  Interdisciplinary Program in Computational Science and Technology\\
  Seoul National University\\
  Seoul, 08826, Republic of Korea\\
  \texttt{dhpark@snu.ac.kr}
}

\date{} 

\begin{document}

\maketitle

\begin{abstract}
We introduce a new class of magic figures defined on a finite triangular region of the trihexagonal tiling. The vertices of the region are labeled with the integers $1,2,\dots,3n(n+1)/2$, each used exactly once. The labeling is called \textbf{trihexagonal magic figure of order $n$} if all triangular faces have the same vertex-sum and all hexagonal faces have the same vertex-sum, with the latter equal to twice the former.

We note three immediate structural features. First, the ratio between the number of independent constraints and the number of variables approaches $1$ as $n\to\infty$, indicating that the system is asymptotically tight. Second, a necessary parity condition arises: since the labeling uses the integers $1,2,\dots,3n(n+1)/2$, the total number of vertices must be odd, which holds if and only if $n \equiv 1$ or $2 \pmod{4}$. Third, the vertex set forms a perfectly regular triangular boundary, so the configuration has an exact polygonal symmetry without boundary irregularities.

\end{abstract}

\section{Introduction}

Magic squares and their variants have long attracted the interest of mathematicians. Beyond squares, many authors have studied other magic figures by making magic labeling on graphs, polygons, stars, hexagons, and triangular arrays. 

Magic figures are a broad class of combinatorial structures generalizing the idea of magic squares, in which labels are assigned to elements of a given shape so that prescribed sums (or analogous quantities) are constant throughout the configuration; this common value is called the magic constant.

In this article we propose a new object: a magic labeling induced by the \textbf{trihexagonal tiling} on a finite triangular region.

Our trihexagonal magic figure is an instance of a combinatorial design problem in which the constraint density increases with $n$ and approaches $1$.

Combined with the presence of parity constraints, this leads to a highly 
restricted system, making the construction of admissible labelings 
significantly more difficult. 

This behavior stands in contrast to many classical magic figures, 
for which constructions are known or widely believed to exist for all orders.

To the best of our knowledge, this type of combinatorial design problem, 
characterized by constraint density approaching $1$, has not been 
explicitly formulated in the context of magic figures, except for the classical square case, where explicit constructions can be found for all orders.

Assuming a random labeling and that the satisfaction of each constraint 
is determined independently, the increasing constraint density suggests 
that admissible labelings become rapidly rare and are expected to disappear 
for sufficiently large $n$. 

In contrast, if solutions persist as $n$ increases, this indicates the presence 
of a highly symmetric underlying structure. In this sense, the problem can be 
interpreted as a combinatorial model in which one can analyze the existence 
of a phase transition in the set of constraint-satisfying configurations.

We analyze the symmetric properties arising in trihexagonal magic figures, 
determine all solutions for $n \le 5$, and show that, under the heuristic 
assumption of random labeling with independently satisfied constraints, 
the expected number of solutions vanishes as $n$ increases.

In many constraint satisfaction problems, increasing constraint density 
is associated with a sharp transition from a satisfiable regime to an 
unsatisfiable one.

However, this heuristic does not necessarily apply to problems in which 
the constraint density approaches $1$ as the order increases.

Moreover, by comparison with the magic quadrangle (A355256),\cite{OEISA355256} 
and most-perfect magic square\cite{OllerenshawBree1998}
we show that having asymptotic constraint density equal to $1$ is not a sufficient 
condition for the disappearance of solutions. 

These examples suggest that the decisive factor is not the asymptotic 
constraint density itself, but the structural compatibility among the 
constraints.

Understanding which combinatorial structures admit solutions for 
infinitely many orders, and which exhibit a genuine satisfiability 
threshold beyond which no solutions exist, requires a systematic study 
of the algebraic and geometric relations among their constraints.

Therefore, a systematic classification of the structural differences 
between families admitting solutions for infinitely many orders and those 
for which solutions eventually disappear, especially when the asymptotic 
constraint density approaches $1$, becomes a necessary direction for 
future research.

\section{Definition of the figure}

\begin{figure}[ht]
    \centering

    \begin{subfigure}{0.14\textwidth}
        \centering
        \includegraphics[width=\linewidth]{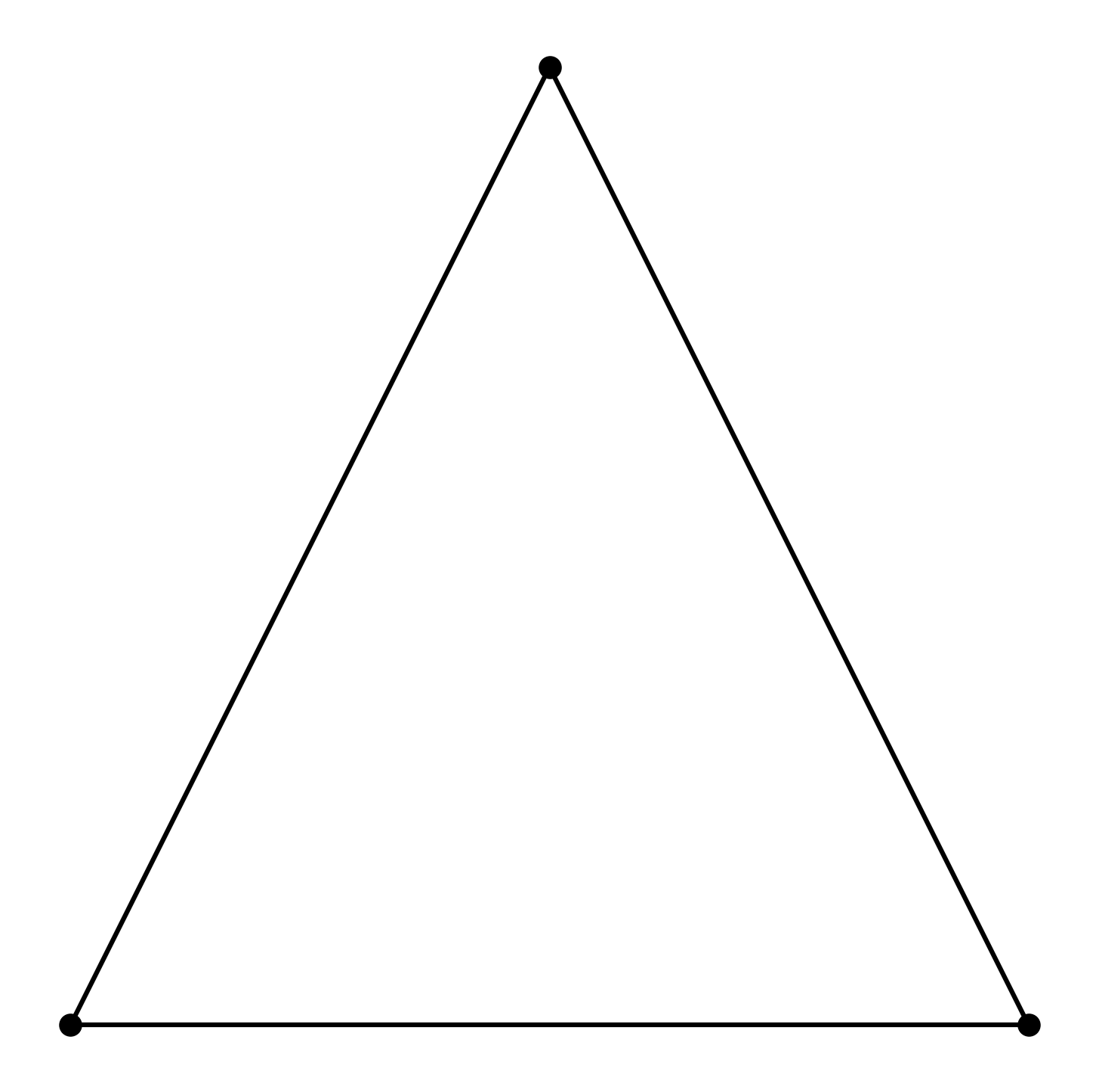}
        \caption{$n=1$}
    \end{subfigure}
    \hfill
    \begin{subfigure}{0.16\textwidth}
        \centering
        \includegraphics[width=\linewidth]{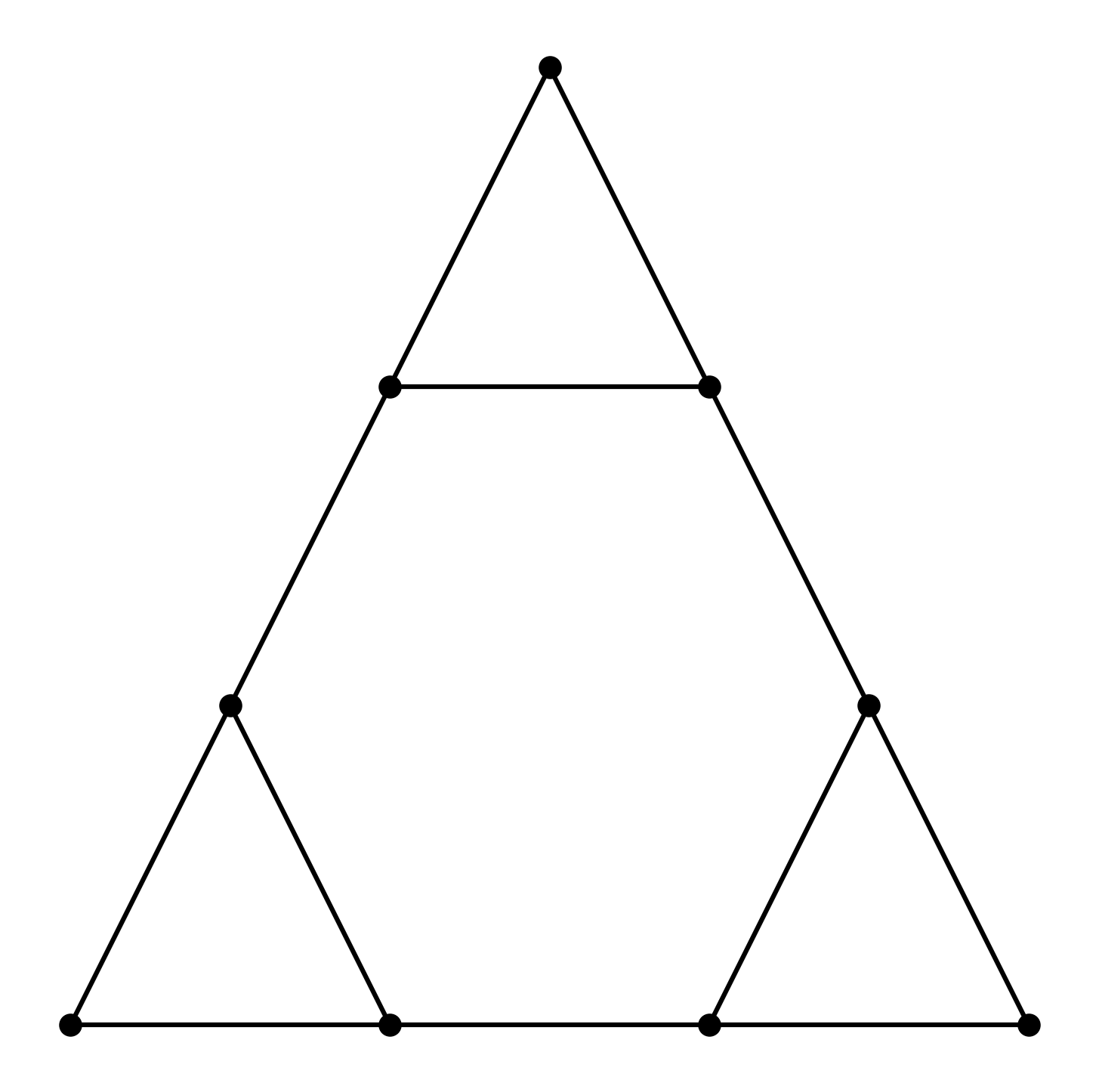}
        \caption{$n=2$}
    \end{subfigure}
    \hfill
    \begin{subfigure}{0.18\textwidth}
        \centering
        \includegraphics[width=\linewidth]{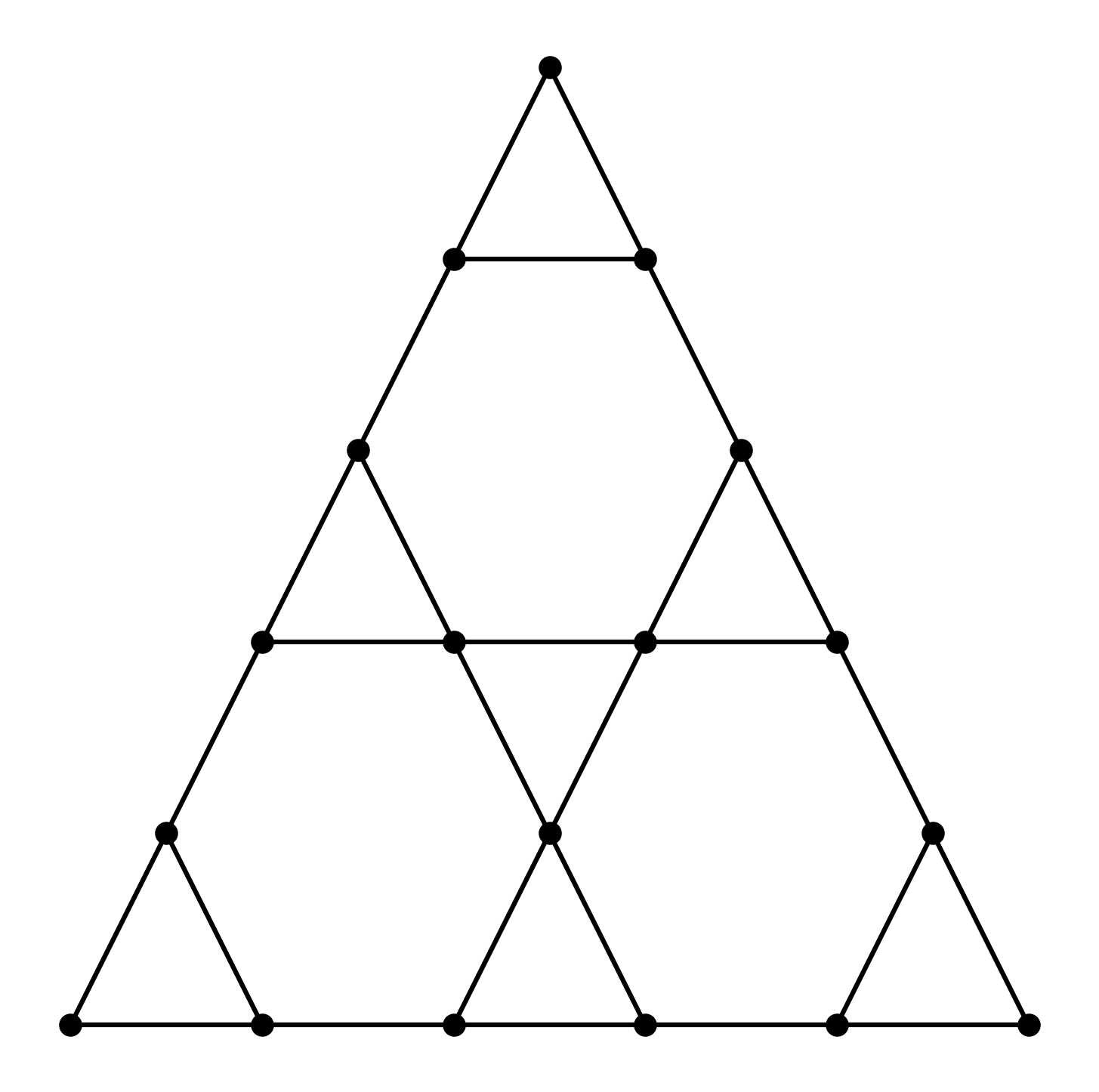}
        \caption{$n=3$}
    \end{subfigure}
    \hfill
    \begin{subfigure}{0.21\textwidth}
        \centering
        \includegraphics[width=\linewidth]{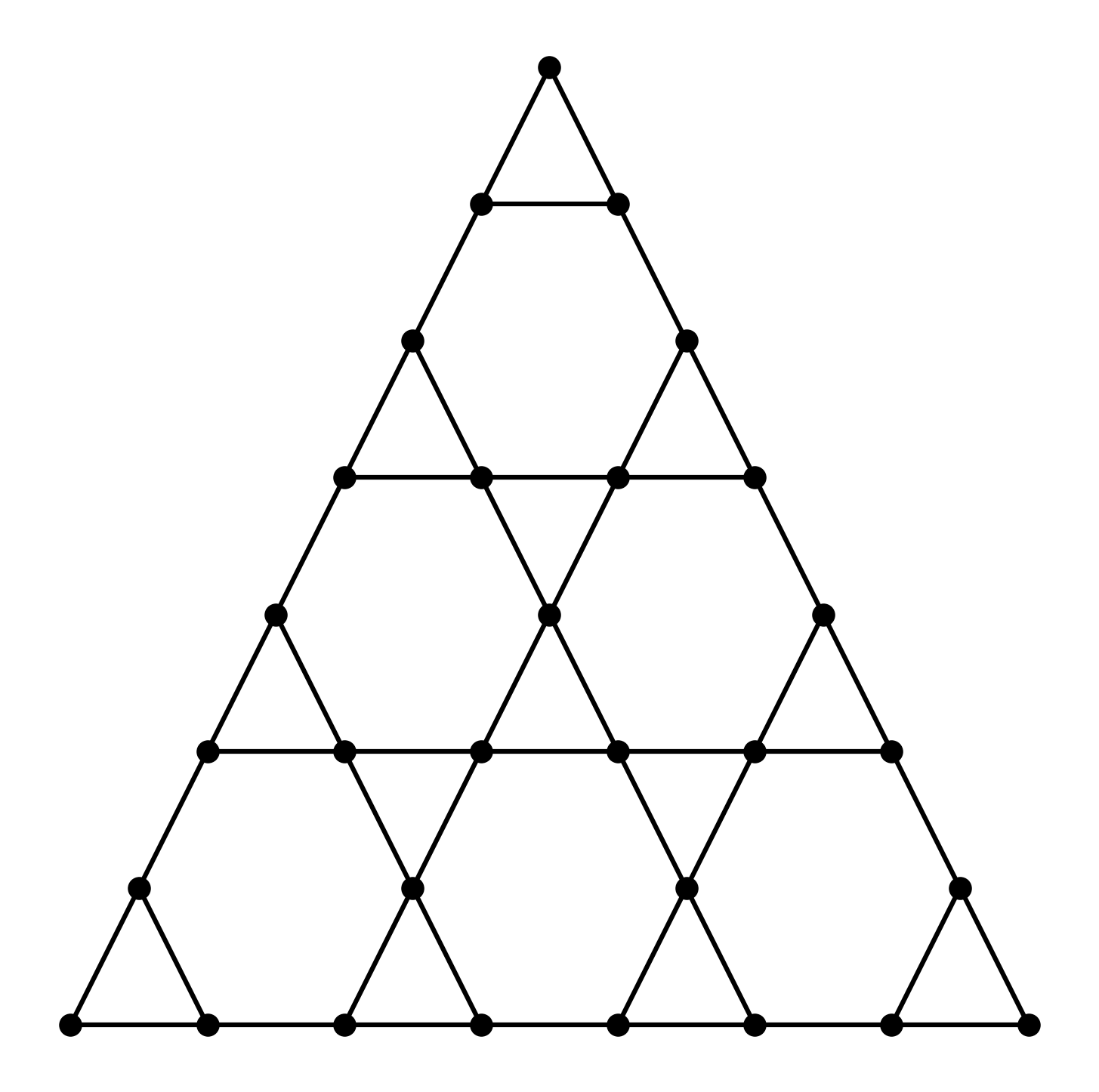}
        \caption{$n=4$}
    \end{subfigure}
    \hfill
    \begin{subfigure}{0.24\textwidth}
        \centering
        \includegraphics[width=\linewidth]{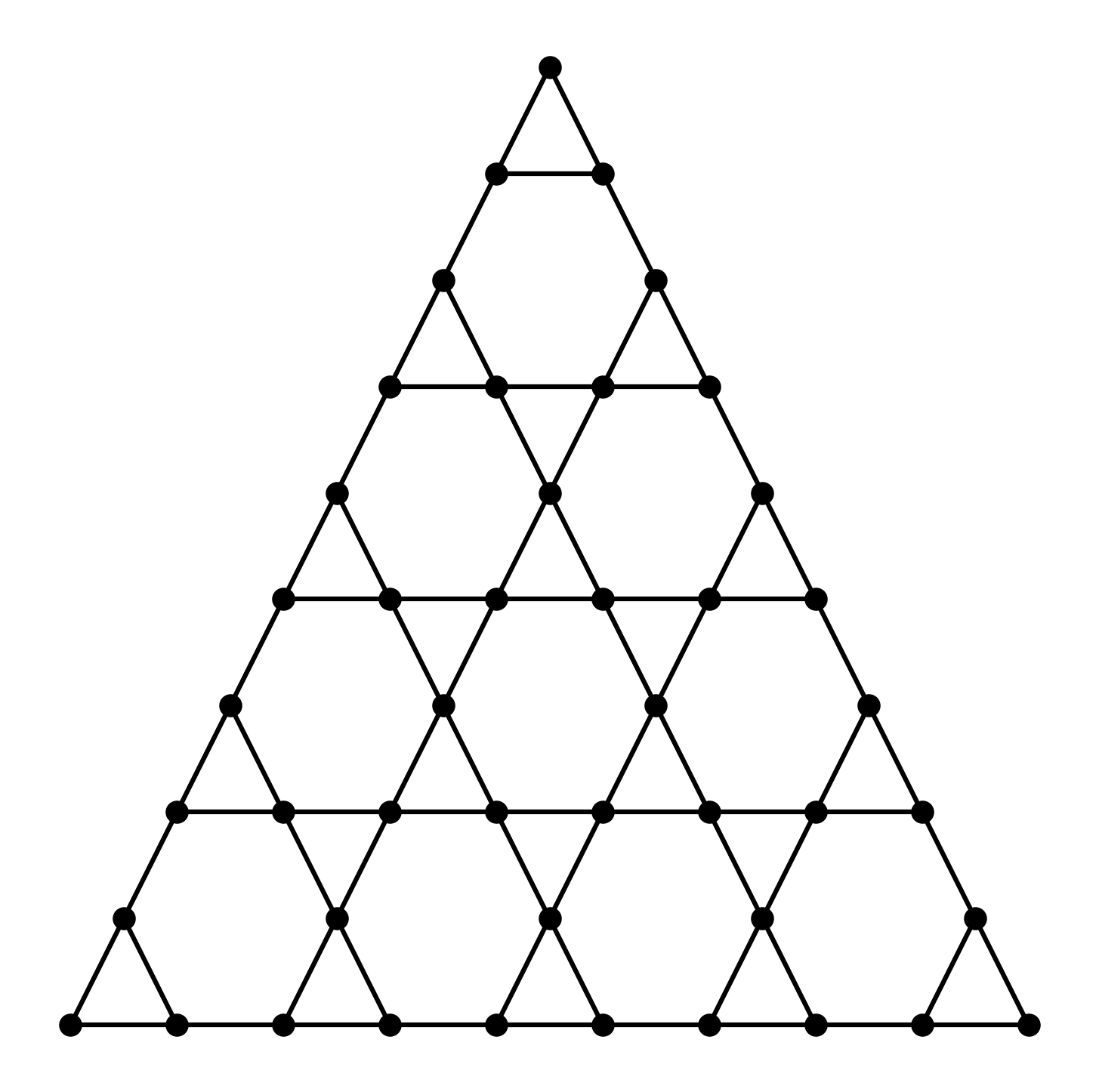}
        \caption{$n=5$}
    \end{subfigure}

    \caption{Construction order from $n=1$ to $n=5$}
    \label{fig:trihexagonal_all}
\end{figure}

An order-n trihexagonal magic figure is obtained by extracting the n-th smallest triangular region from the trihexagonal tiling (so that each side contains 2n vertices) and assigning numbers to its vertices.

Consider an equilateral triangle whose sides are subdivided into \(2n-1\) equal segments.
From each vertex, select the two adjacent sides and mark points at odd distances \(1,3,5,\dots,2n-1\) from the vertex.
For each such point, draw a line parallel to the side opposite the chosen vertex.

Performing this construction independently at all three vertices, the intersection of the resulting families of parallel lines produces precisely the trihexagonal region of order \(n\).

A trihexagonal magic figure is then obtained by assigning numbers to the intersection points where two of these three families of parallel lines meet (including those on the boundary of the triangle), so that all prescribed magic sums are equal.

Therefore, in an order-\(n\) trihexagonal region, there are exactly \(n\) upward-pointing unit triangles along each side of the outer boundary, and every vertex belongs to exactly one such upward-pointing triangle.

Equivalently, an order-\(n\) trihexagonal region can be obtained from a triangular arrangement of \(\binom{2n+1}{2}\) points with \(2n\) vertices on each side, by removing \(\binom{n}{2}\)   vertices, each of which corresponds to a hexagon in the trihexagonal tiling.

This yields
\[
|V_n| = \frac{3n(n+1)}{2}.
\]

at the number of vertices.

A \textbf{trihexagonal magic figure of order $n$} is a bijection
\[
\varphi:V_n\to\{1,2,\dots,\tfrac{3n(n+1)}{2}\}
\]
such that

\begin{itemize}
    \item Every triangular face has the same sum $T_n$ of its three vertex labels;
    \item Every hexagonal face has the same sum $H_n$ of its six vertex labels.
\end{itemize}
Because each hexagon is naturally decomposable into two alternating triples compatible with the surrounding triangular structure, the intended magic condition is
\[
H_n=2T_n.
\]

Two such labelings are considered equivalent when one is carried to the other by a symmetry of the outer triangular region, that is, by an element of the dihedral group $D_3$ generated by rotations and reflections.

\subsection{Why trihexagonal?}
We focus on the trihexagonal tiling because it enables the formulation of magic figure problems with high constraint density, while at the same time admitting finite regions whose boundary forms an exact regular polygon, yielding a particularly natural and aesthetically coherent configuration.

To motivate this choice, we begin by examining the asymptotic constraint density as the number of tiles tends to infinity.

This can be derived directly from the angle condition for Euclidean tilings. 
If regular polygons with side-sharing meet at an interior vertex, and their 
numbers of sides are $n_1,\dots,n_k$, then
\[
\sum_{i=1}^k \left(\pi - \frac{2\pi}{n_i}\right) = 2\pi,
\]
hence
\[
\sum_{i=1}^k \frac{1}{n_i} = \frac{k-2}{2}.
\]
Since each $n_i$-gon contributes one constraint distributed over its $n_i$ vertices, 
the average number of constraints per vertex is exactly
\[
\sum_{i=1}^k \frac{1}{n_i} = \frac{k-2}{2}.
\]
Therefore, tilings with three polygons meeting at each vertex have asymptotic 
constraint density $1/2$, those with four have density $1$, and those with five 
or more have density greater than $1$.

Moreover, when five or more regular polygons meet at a vertex, it follows from the angle condition that at least three of them must be triangles, so that triangles constitute a majority of the faces around that vertex. In such a configuration, the tiling necessarily contains pairs of adjacent triangles sharing an edge, which is incompatible with the present type of magic labeling. Consequently, the maximal feasible constraint density occurs when exactly four polygons meet at each vertex, and among uniform tilings this situation is realized only by the square tiling, trihexagonal tiling and rhombitrihexagonal tiling.

Among the eleven uniform tilings of the plane, only a few admit the property that finite subsets of tiles can be selected so that their boundary forms an exact regular polygon, and this remains possible as the size of the polygon grows. In particular, this condition is satisfied by the triangular tiling, the square tiling, and the trihexagonal tiling.

However, as discussed above, tilings in which five or more polygons meet at each vertex are incompatible with the present type of magic labeling, and thus the triangular tiling is excluded.

We therefore focus on triangular regions of the trihexagonal tiling. Although regular polygonal regions can be extracted in both triangular and hexagonal shapes in this tiling, the hexagonal case leads to configurations in which the parity of the number of vertices obstructs the required averaging condition, making the magic constraint unattainable. This justifies our restriction to triangular regions.

\section{Preliminaries}

We define the number of naive constraints as the number of conditions imposed by the definition of a magic figure, namely the requirement that all designated faces or lines have equal magic sums.

For example, a magic square of order $n$ has $2n+2$ naive constraints, corresponding to the $n$ row sums, $n$ column sums, and two diagonal sums.

\begin{proposition}\label{prop:constraints}
An order-$n$ trihexagonal magic figure contains $\binom{n+1}{2}$ upright-triangular faces, $\binom{n}{2}$ hexagonal faces, and $\binom{n-1}{2}$ inverted triangular faces,
\end{proposition}

\begin{proof}
The upright-triangular faces correspond to the unit triangles in a triangular arrangement of side length $n$, which are counted by $\binom{n+1}{2}$. 

The hexagonal faces correspond to the interior positions where six triangles meet, forming a smaller triangular arrangement of side length $n$, and are therefore counted by $\binom{n}{2}$.

The inverted triangular faces are located between adjacent hexagons arranged in a triangular pattern, and are therefore also counted by $\binom{n-1}{2}$.
\end{proof}

\begin{proposition}
For an order-$n$ trihexagonal magic figure, the common triangular magic sum is equal to three times the average of all vertex labels.
\end{proposition}

\begin{proof}
Consider only the $\binom{n+1}{2}$ upright triangular faces. These faces are pairwise disjoint, and their union contains all vertices of the figure exactly once. Hence the sum of their magic sums is equal to the sum of all vertex labels:
\[
\binom{n+1}{2} T_n = \sum_{k=1}^{3n(n+1)/2} k.
\]
Since
\[
\sum_{k=1}^{3n(n+1)/2} k
=
\frac{3n(n+1)}{2}\cdot \frac{\frac{3n(n+1)}{2}+1}{2},
\]
and
\[
3\binom{n}{2}=\frac{3n(n+1)}{2},
\]
it follows that
\[
T_n
=
3\cdot \frac{\frac{3n(n+1)}{2}+1}{2}.
\]
In other words, the common triangular magic sum is three times the average of all vertex labels.
\end{proof}

\begin{proposition}
A necessary condition for the existence of an order-$n$ trihexagonal magic figure is that
$\binom{n+1}{2}$ or the number of vertices $3 \binom{n+1}{2}$
be odd. Equivalently, this holds only when
\[
n \equiv 1 \text{ or } 2 \pmod{4}.
\]
\end{proposition}

\begin{proof}
By the previous proposition, the common triangular magic sum is equal to three times the average of all vertex labels. Since the vertex labels are
\[
1,2,\dots,\frac{3n(n+1)}{2},
\]
their average is
\[
\frac{\frac{3n(n+1)}{2}+1}{2}.
\]
Hence the triangular magic sum is an integer only if this average is an integer, that is, only if
\[
\frac{3n(n+1)}{2}
\]
is odd. Since $3$ is odd, this is equivalent to requiring that
\[
\frac{n(n+1)}{2}=\binom{n+1}{2}
\]
be odd.

It remains to determine when $\binom{n+1}{2}$ is odd. This occurs exactly when
\[
n \equiv 1 \text{ or } 2 \pmod{4}.
\]
Indeed, if $n\equiv 0$ or $3\pmod{4}$, then $n(n+1)$ is divisible by $4$, so $\frac{n(n+1)}{2}$ is even. On the other hand, if $n\equiv 1$ or $2\pmod{4}$, then $n(n+1)\equiv 2\pmod{4}$, so $\frac{n(n+1)}{2}$ is odd.
\end{proof}

\section{Local structure, propagation, and invariants}

We now turn to the actual structure of the trihexagonal configuration and the symmetries exhibited by Figure~\ref{fig:hexagon_propagation}. 
Our strategy is to first analyze the local behavior of the smallest building blocks. 
In particular, we begin by establishing the fundamental properties of unit triangles and unit hexagons (that is, those of side length $1$), which will serve as the basis for all subsequent generalizations.

\begin{proposition}[Local determinacy and propagation]
\
\begin{enumerate}
    \item Any triangle is uniquely determined by two of its vertices.
    \item Any hexagon is uniquely determined by five of its vertices.
    \item If four consecutive vertices of a hexagon are known, then the vertex of the opposite triangle is uniquely determined.
\end{enumerate}
\end{proposition}

\begin{figure}[ht]
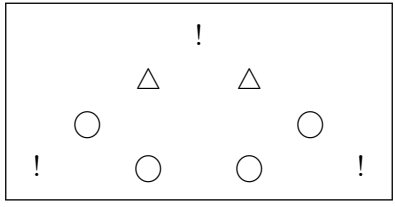

\centering

\setlength{\arraycolsep}{6pt}
\renewcommand{\arraystretch}{1.2}
\setlength{\fboxsep}{4pt}
\setlength{\fboxrule}{0.5pt}

\fbox{$
\begin{array}{ccccccc}
 &   &   & !      &   &   &   \\
 &   & \triangle &        & \triangle &   &   \\
 & \bigcirc &   &    &   & \bigcirc &   \\
 ! &   & \bigcirc &  & \bigcirc &   & !
\end{array}
$}

\caption{
Local determinacy and propagation.
(1) Two vertices of a triangle (marked by $\bigcirc$) determine the third vertex (marked by $!$).
(2) Four consecutive vertices of a hexagon (marked by $\bigcirc$) together with one adjacent vertex (marked by $\triangle$) determine the remaining adjacent vertex (marked by $\triangle$).
(3) Consequently, four consecutive vertices of a hexagon determine the vertex of the opposite triangle (marked by $!$).
}

\label{fig:hexagon_propagation}
\end{figure}

\begin{proof}
Parts (i) and (ii) follow directly from the defining constraints of the configuration. 
For (iii), although the remaining two vertices of the hexagon are not individually determined, the triangle containing both of them is uniquely determined, which fixes its third vertex.
\end{proof}

\begin{remark}[Lattice polygons and structural properties]
We now introduce general lattice polygons in the trihexagonal configuration and record several structural properties that will be used without proof.

In the trihexagonal figure, one can define polygons such as triangles, hexagons, and parallelograms whose sides lie along lattice directions (i.e., parallel to the edges of the outer triangular region).

\begin{itemize}
    \item[(1)] In an order-$n$ configuration, the largest triangle has side length $2n-1$ (measured in lattice steps), and contains all vertices of the configuration.

    \item[(2)] The configuration contains, for each $1 \le m \le n$, a trihexagonal subfigure of order $m$. 
    Such a subfigure corresponds to a triangular region of side length $2m-1$. 
    In particular, all triangles arising in this way have odd side length.

    \item[(3)] Hexagons can also be defined in the same manner, with all six sides lying along lattice directions. 
    For any such hexagon, all side lengths are necessarily odd.

    \item[(4)] Any parallelogram whose sides lie along lattice directions has interior angles $60^\circ$ and $120^\circ$. 
    Moreover, the side lengths of such parallelograms are necessarily even.
\end{itemize}

These properties follow from the combinatorial structure of the trihexagonal lattice and will be used without further justification.
\end{remark}

\begin{theorem}[Global rigidity of lattice polygons]
In a trihexagonal magic figure, the local triangle constraints imply global conservation laws on all lattice polygons:
\begin{itemize}
    \item Every lattice triangle has the same vertex-sum; (Lemma~\ref{lem:triangle-sum-invariance})
    \item Every lattice hexagon has the same vertex-sum equal to twice the triangle sum; (Proposition ~\ref{prop:triangle-sum-invariance})
    \item Every lattice parallelogram satisfies equality of opposite-vertex sums. (Therorem~\ref{thm:parallelogram-invariance})
\end{itemize}
The following lemma and propositions establish these statements successively.
\end{theorem}

\begin{lemma}[Triangle-sum invariance]\label{lem:triangle-sum-invariance}
Let $T$ be an $N \times N$ triangular configuration, and let $T'$ be any $M \times M$ subtriangle contained in $T$, either upright or inverted. Then the sum of the values at the three vertices of $T'$ is equal to the sum of the values at the three vertices of $T$.
\end{lemma}

\begin{proof}
Let $\mu$ be the average of all values on the vertices of $T$.
By definition, the sum of the values on any triangle is
\[
T_n = 3\mu,
\]
and the sum of the values on any hexagon is
\[
H_n = 6\mu.
\]

We first show that the sum of the values at the three vertices of the whole $N\times N$ triangle is also $3\mu$.

The total number of vertices in $T$ is
\[
3\binom{N+1}{2}.
\]
Hence the total sum of all vertex values is
\[
3\binom{N+1}{2}\mu.
\]

Now classify the vertices of $T$ as follows:
\begin{itemize}
    \item each vertex on an edge but not a corner belongs to exactly one hexagon;
    \item each interior vertex belongs to exactly two hexagons;
    \item each interior vertex belongs to exactly one inverted triangle.
\end{itemize}

The number of hexagons is
\[
\binom{N}{2},
\]
and the number of inverted triangles is
\[
\binom{N-1}{2}.
\]

Therefore, if we subtract the sums over all hexagons and add the sums over all inverted triangles, then every non-corner vertex is cancelled exactly once, while each corner vertex remains once. It follows that the sum of the three corner values of $T$ is
\[
3\binom{N+1}{2}\mu
-6\binom{N}{2}\mu
+3\binom{N-1}{2}\mu.
\]
A direct simplification gives
\[
3\binom{N+1}{2}\mu
-6\binom{N}{2}\mu
+3\binom{N-1}{2}\mu
=3\mu.
\]
Thus the sum of the three vertices of the whole triangle is $3\mu$, namely $T_n$.

The same argument applies to any upright or inverted $M\times M$ subtriangle $T'$, since it is itself a triangular configuration of order $M$. Hence the sum of the values at the three vertices of $T'$ is also $3\mu$.

Therefore, every subtriangle has the same vertex-sum as the whole triangle.
\end{proof}

\begin{figure}[ht]
    \centering
    \includegraphics[width=0.6\textwidth]{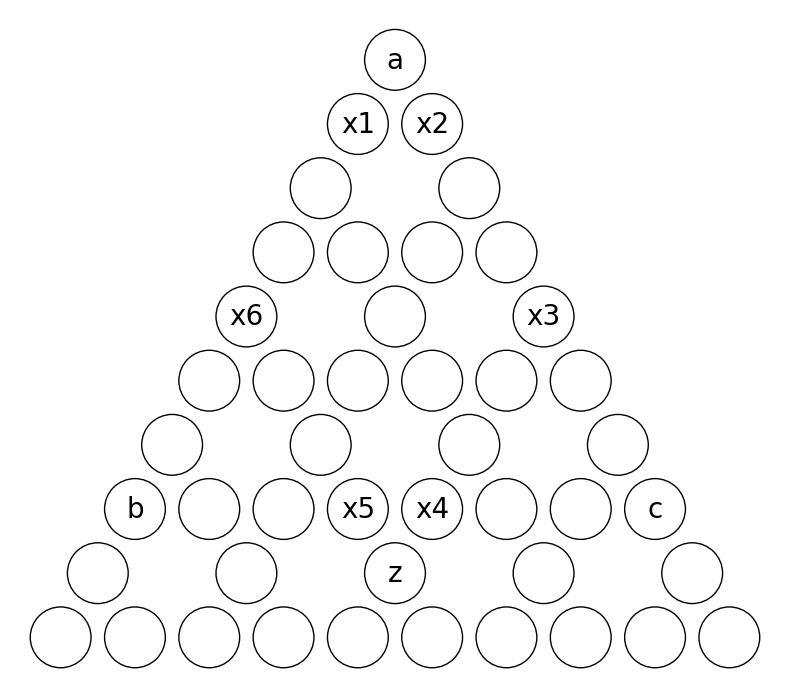}
    \caption{Illustration of the construction order}
    \label{fig:examplepolygon}
\end{figure}

\begin{proposition}[Hexagon-sum invariance]\label{prop:triangle-sum-invariance}
Let $H$ be any hexagon formed by six vertices in the triangular configuration,
whose edges lie along straight lines connecting adjacent nodes of the lattice.
Then the sum of the values at its vertices is constant,
independent of the choice of the hexagon.
\end{proposition}

\begin{proof}
Let the six vertices of $H$ be denoted, in cyclic order, by
\[
x_1,x_2,x_3,x_4,x_5,x_6.
\]
(See Figure~\ref{fig:examplepolygon} for a representative configuration of this construction.)

Extend the sides
\[
x_2x_3,\qquad x_4x_5,\qquad x_6x_1
\]
until the three resulting lines meet pairwise.
Denote these three intersection points by $a,b,c$.
Then $a,b,c$ form an upright triangle containing $H$.

Moreover, the nine vertices
\[
a,b,c,x_1,x_2,x_3,x_4,x_5,x_6
\]
can be partitioned into three triangles,
for example
\[
(a,x_1,x_2),\qquad (b,x_5,x_6),\qquad (c,x_3,x_4),
\]
each of whose sides lies along lattice lines.

By Lemma~\ref{lem:triangle-sum-invariance}, each of these three triangles has vertex-sum $T_n$.
Hence
\[
(a+x_1+x_2)+(b+x_5+x_6)+(c+x_3+x_4)=3T_n.
\]
Since $a,b,c$ themselves form a triangle, we also have
\[
a+b+c=T_n.
\]
Subtracting gives
\[
x_1+x_2+x_3+x_4+x_5+x_6 = 2T_n.
\]
Therefore the hexagon-sum is invariant.

\end{proof}

\begin{theorem}[Rhombus or Parallelogram invariance]\label{thm:parallelogram-invariance}
Consider an $M \times N$ rhombus or parallelogram-shaped region in the triangular configuration, regardless of its orientation. 
This region consists of $MN$ hexagons and $2MN$ triangles, andside of the rhombus contains $2M+1$ or $2N+1$ vertices, including the endpoints.

Fix such a rhombus. Then the sums of the values at opposite vertices of the rhombus are equal.
\end{theorem}

\begin{proof}
Let $P$ be the given parallelogram in the triangular lattice,
whose edges lie along lattice directions.

In Figure~\ref{fig:examplepolygon}, the vertices $a$ and $z$ are the two opposite vertices with acute angles, 
while $x_3$ and $x_6$ are the two opposite vertices with obtuse angles.

Choose an acute vertex of $P$, and consider a smaller triangle based at this vertex 
whose side length is strictly smaller than that of the shorter side of $P$. 
Using this triangle, we construct a hexagon $H$ inscribed in $P$ whose edges follow lattice lines. 
In Figure~\ref{fig:examplepolygon}, this corresponds to selecting the pairs $(x_1,x_2)$ at $a$ and $(x_4,x_5)$ at $z$. 
The same construction is performed at both acute vertices of $P$.

This construction introduces four additional vertices.
Together with the two vertices of one diagonal of $P$ (the shorter diagonal),
these six vertices form a hexagon.
Hence their sum is $H_n$.

On the other hand, the same four vertices together with the two vertices of the other diagonal
(the longer diagonal) decompose into two triangles.
Hence their total contribution is $2T_n$.

Since $H_n = 2T_n$, the sums of the values at the two pairs of opposite vertices coincide.
Therefore, the sums along both diagonals of $P$ are equal.
\end{proof}

\section{Enumeration and existence for small orders}

For the purpose of enumeration, we follow the standard convention used for magic squares and count trihexagonal magic figures up to the dihedral symmetries of the outer triangle, namely rotations and reflections.

For $n=1$, the configuration is trivial, and there is a unique solution up to symmetry. Hence the number of solutions is $1$.

For $n=2$, the configuration consists of placing the numbers $1$ through $9$ on a triangular arrangement of side length $4$, with the central position removed. The three vertices of the outer triangle determine the structure, and fixing their relative order allows one to enumerate all solutions directly.

A straightforward enumeration shows that there are $80$ solutions up to symmetry.

\paragraph{Divide-and-conquer strategy for n=5}

We describe a divide-and-conquer procedure for determining the central part of the configuration.

We begin by assigning values to the six vertices adjacent to the central hexagon in the $45$-node triangular configuration. 
These vertices form a star-shaped pattern consisting of one upright triangle and one inverted triangle. 
The sum of the values on each of these triangles is equal to the magic sum $T_n$ of a unit triangle (see Lemma~\ref{lem:triangle-sum-invariance}).

\begin{figure}[ht]
    \centering
    \includegraphics[width=0.6\textwidth]{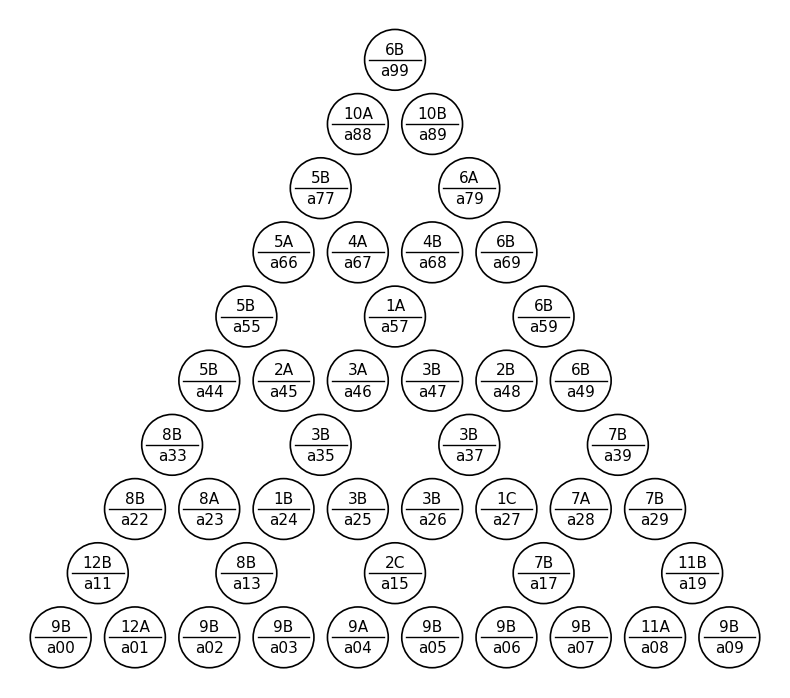}
    \caption{Illustration of the construction order}
    \label{fig:sunseo}
\end{figure}

To avoid counting configurations equivalent under rotations and reflections, we impose an ordering condition and first assign values to the upright triangle so that its sum is $69$. This fixes a canonical representative. We then determine the values on the inverted triangle accordingly.

Once the value at one of the six vertices adjacent to the central hexagon is fixed, all remaining vertices of the central hexagon are uniquely determined by the local propagation rules.

After fixing the central configuration, the remaining vertices are determined sequentially by repeated application of the local propagation rules.

At each step, once a vertex is assigned, the values of adjacent vertices are uniquely determined. This dependency is illustrated in Figure~\ref{fig:sunseo}, where arrows indicate the direction of determination.

Thus, the entire configuration is uniquely determined from the choices at each step.

We now explain the labeling used in Figure~\ref{fig:sunseo}.

In the first two steps, labeled $1$ and $2$, the symbols $A$, $B$, and $C$ indicate the order in which the values of the corresponding triangles are assigned. 
These steps each have two degrees of freedom.

From the subsequent steps onward, the labeling reflects the propagation structure. 
If two vertices share the same index $n$, then the values at all vertices labeled $nB$ are uniquely determined once the value at the vertex labeled $nA$ is fixed. 
In other words, each step $n \ge 3$ introduces a single degree of freedom, and all remaining vertices with the same index are determined simultaneously by propagation.

We next describe the behavior of the boundary vertices.

The three corner vertices of the large triangle are determined before the values at the vertices adjacent to them are fixed. 
Indeed, they are determined via the propagation rule that if four consecutive vertices of a hexagon are known, then the vertex of the opposite triangle is uniquely determined.

In contrast, the six vertices adjacent to the three corners are determined at the final stage. 
These six vertices form three disjoint adjacent pairs. Within each pair, the two values can be interchanged without affecting the validity of the configuration.

This symmetry holds for all $n \ge 2$. Consequently, the number of distinct configurations (up to the imposed ordering) is always divisible by $2^3 = 8$ for $n \ge 2$.

\begin{figure}[ht]
    \centering
    \includegraphics[width=0.6\textwidth]{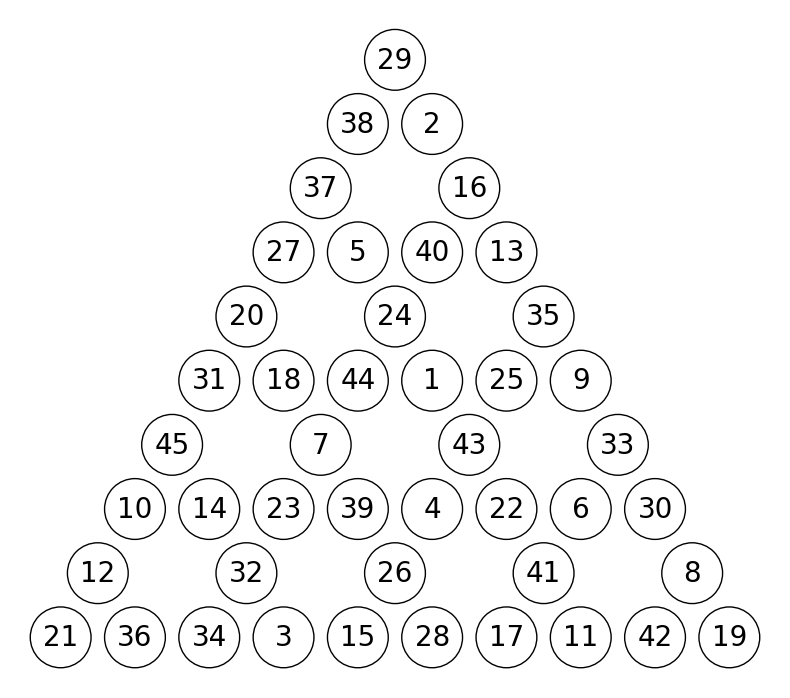}
    \caption{Illustration of the example solution at n=5}
    \label{fig:solution5}
\end{figure}

For \(n=5\), our enumeration yields \(6976\) base solutions, inequivalent up to the swaps of the three certain adjacent vertex pairs described above. As each pair may be swapped independently, every base solution gives rise to \(2^3=8\) distinct magic labelings. Hence the total number of magic labelings is
\[
6976 \cdot 8 = 55808.
\]

The complete enumeration source code used for the cases \(n=2\) and
\(n=5\) is included in the Appendix for reproducibility, and the original
programs were previously published in the author’s public GitHub
repository.

One example of an order-\(5\) trihexagonal magic labeling is displayed in Figure~\ref{fig:solution5}.

\paragraph{Heuristic asymptotic estimate}

We conclude with a rough heuristic estimate for the growth of the number of solutions as \(n\) increases.

Let \(N = \frac{3n(n+1)}{2}\) be the number of vertices. The labeling is a permutation of \(\{1,2,\dots,N\}\), so the total number of assignments is \(N!\).

On the other hand, the magic conditions impose a system of linear constraints. 
Let \(C_n\) denote the number of (naive) constraints, that is, the number of triangular and hexagonal face conditions. From Proposition~\ref{prop:constraints}, we have
\[
C_n = \binom{n+1}{2} + \binom{n}{2} + \binom{n-1}{2}
= \frac{3n^2 - 3n + 2}{2}.
\]

Heuristically, if each constraint reduces the number of degrees of freedom by one and behaves independently, one may expect the number of solutions to be on the order of
\[
\frac{N!}{N^{C_n}}.
\]

Since
\[
N = \frac{3n(n+1)}{2}, \quad
C_n \sim \frac{3n^2}{2},
\]
the ratio between the number of constraints and the number of variables satisfies
\[
\frac{C_n}{N} \to 1 \quad \text{as } n \to \infty.
\]

This suggests that the system is asymptotically tight, and that the number of solutions grows much more slowly than \(N!\), potentially exhibiting a transition to non-existence beyond a certain size. We leave a rigorous analysis of this phenomenon as an open problem.

Phase transition phenomena are well known in combinatorics:
as a density or probability parameter crosses a critical range, a
discrete structure may abruptly change its typical properties.
\cite{JansonLuczakRucinski2000,Friedgut1999,MolloyReed1995}

We suggest that an analogous phenomenon may occur even when the
constraint density is not random but increases according to a deterministic
rule with the order. In such systems, both the existence of solutions and
the difficulty of construction may change sharply beyond a certain
threshold.

\begin{table}[ht]
\centering
\caption{Heuristic estimates and known data for small orders}
\label{tab:heuristic}
\begin{tabular}{cccccc}
\toprule
\(n\) & \(N\) & \(C_n\) 
& \(\displaystyle \frac{N!}{N^{C_n}}\) 
& \(\displaystyle \frac{N!}{N^{C_n-1}}\) 
& solutions\footnotemark \\
\midrule
1  & 3   & 1   & \(2\)                     & \(6\)                     & \(1\) \\
2  & 9   & 4   & \(55.31\)                & \(497.78\)               & \(80\) \\
3  & 18  & 10  & \(1.793\times 10^3\)     & \(3.228\times 10^4\)     & \(\varnothing\) \\
4  & 30  & 19  & \(2.282\times 10^4\)     & \(6.847\times 10^5\)     & \(\varnothing\) \\
5  & 45  & 31  & \(6.733\times 10^4\)     & \(3.030\times 10^6\)     & \(55808\) \\
6  & 63  & 46  & \(3.370\times 10^4\)     & \(2.123\times 10^6\)     & --- \\
7  & 84  & 64  & \(2.325\times 10^3\)     & \(1.953\times 10^5\)     & \(\varnothing\) \\
8  & 108 & 85  & \(1.910\times 10^1\)     & \(2.063\times 10^3\)     & \(\varnothing\) \\
9  & 135 & 109 & \(1.673\times 10^{-2}\)  & \(2.258\)                & --- \\
10 & 165 & 136 & \(1.434\times 10^{-6}\)  & \(2.366\times 10^{-4}\)  & --- \\
\bottomrule
\end{tabular}
\end{table}

\footnotetext{Solutions are counted up to the dihedral symmetries of the outer triangle. 
If rotations and reflections are distinguished, the total number of labelings is six times larger.}

\paragraph{Comparison with classical magic figures}

It is natural to compare the present construction with previously studied magic figures arising from Euclidean tilings. In particular, configurations based on tilings in which three polygons meet at each vertex have asymptotic constraint density \(1/2\).

Representative examples include Choi Seok-jeong's \emph{Jisugwimundo}, which is naturally associated with a hexagonal lattice,\cite{ChoeChoiMoon2003} and Yang Hui's \emph{yeonhwando},\cite{Park2016} which may be interpreted in relation to the truncated square tiling(4.8.8). Both illustrate that classical magic figures already arise from Euclidean tiling structures, although the tilings and local constraint patterns differ substantially from the trihexagonal case considered here.

In contrast, the trihexagonal tiling considered in this paper has asymptotic constraint density \(1\), which is the maximal value among admissible uniform tilings compatible with the present type of labeling. This substantial increase in constraint density leads to a much more rigid system, as reflected in the rapid decrease of the heuristic estimates and the apparent disappearance of solutions for certain values of \(n\).

We also remark that a closely related idea appears implicitly in constructions based on the square tiling. Although there is no universally established terminology, one may consider labelings of the vertices in which all square faces have the same vertex-sum, i.e., a face-magic labeling on the square grid. The enumeration of such square-grid configurations is recorded in OEIS A355256 \cite{OEISA355256}.

From another perspective, the configurations counted by OEIS A355256 may be regarded as square analogues of the \emph{Jisugwimundo}, since both can be viewed as face-magic labelings of regular tilings.

While this viewpoint has not been widely studied as an independent subject, it is naturally realized in classical objects such as most-perfect magic squares \cite{OllerenshawBree1998}, which satisfy strong local sum constraints on subsquares.

Square-based configurations admit a much more flexible global structure. 
Therefore, the fact that a naive estimate obtained from the number of local
constraints tends to zero does not imply the actual disappearance of solutions.

A simple counting heuristic further illustrates this discrepancy. For instance,
one obtains an estimate of the form
\[
\frac{(n^2)!}{(n^2)^{(n-1)^2}},
\]
which already falls below 1 for moderate values (e.g., around $n=8$ or $9$).
Taken at face value, this would suggest that solutions should not exist.

However, this heuristic ignores substantial structural degeneracies in the
square-based setting. In particular, once a solution is constructed, additional
solutions can be generated by permuting rows within the same parity class:
odd-indexed rows may be permuted among themselves, and even-indexed rows
likewise. These symmetries alone already produce a large family of distinct
solutions, demonstrating that the naive estimate significantly undercounts the
true number.

In Appendix, we present explicit constructions illustrating the existence of
such solutions.

Beyond the existence of solutions for OEIS~A355256, an even stronger
example is provided by most-perfect magic squares. These are known to
exist for every doubly-even order \(n\) (that is, whenever \(4\mid n\))
\cite{OllerenshawBree1998}.

This shows that the fact that a heuristic enumeration estimate based on
constraint density tends to zero under a random-array assumption does
not necessarily imply a phase transition in the actual solution space.
Even under substantially stronger local and global constraints,
solutions may still persist for infinitely many orders when sufficient
algebraic structure is present.

The same local-sum principle of OEIS~A355256 can also be extended from square arrays to
triangular ones.

As a further example, consider the configuration counted by OEIS~A375416,\cite{OEISA375416}.
namely order-$n$ magic triangles with $n^2$ unit triangles(nodes) in which all $2\times2$ subtriangles have a
common sum. This again constitutes a region-based constraint system with a
high degree of symmetry, and the number of constraints grows quadratically
with $n$, leading to an asymptotic constraint density approaching 1.

Such examples suggest a broader question: given a family of combinatorial
configurations arising from a regular tiling (e.g., triangular, square, or
mixed tilings), and imposing uniform sum conditions on local regions, does
the system admit solutions for arbitrarily large $n$?

The evidence presented here indicates that the answer depends sensitively on
the interaction structure of the constraints. Even when the number of
constraints is comparable to the number of variables, and the underlying
geometry is highly symmetric, the resulting systems may either admit
infinitely many solutions or become infeasible beyond a certain threshold.

From the perspective of combinatorial design, this highlights the importance
of understanding not only the density of constraints but also their geometric
and algebraic compatibility. In particular, the transition between persistent
existence and eventual nonexistence appears to be governed by subtle structural
properties of the underlying tiling.
\subparagraph{Attribution and originality.}

The square-based family A355256 is not a newly introduced object: its
underlying local-sum structure already appears implicitly in classical
square-based magic figures in the Western literature, and it was later
formulated explicitly in Dairyong Jeong’s \textit{Susja Nori}\cite{JeongSite,JeongQuadrangle,JeongOdd,JeongEven}, now no
longer publicly maintained.

By contrast, the systems A375416 and A391236 are original constructions
introduced by the author as independent combinatorial models for studying
phase transitions under asymptotically balanced constraints. Their
enumeration sequences were previously recorded in OEIS, while the present
work provides the mathematical formulation and interpretation of these
families in the context of combinatorial design.

\subparagraph{Constraint density and heuristic phase transition.}
We first observe that, under the naive random-labeling heuristic with 
independently satisfied constraints, the asymptotic constraint density 
determines two qualitatively different regimes.

When the asymptotic constraint density approaches $1$, the expected number 
of admissible labelings tends to zero, suggesting the disappearance of 
solutions for sufficiently large orders.

In contrast, when the asymptotic constraint density remains strictly below 
$1$, the same heuristic predicts that admissible labelings do not vanish 
asymptotically, since a positive proportion of effective degrees of freedom 
remains.

This distinction is formalized in the following two propositions.

\begin{proposition}[Naive Overconstraint Collapse]
Let $k \ge 1$ and $\alpha > 0$ be fixed constants. 
Suppose the node count $N(n)$ and the naive constraint count $C(n)$ satisfy
\[
N(n) = \alpha n^k + O(n^{k-1}), 
\qquad
C(n) = \alpha n^k + O(n^{k-1}).
\]
Then
\[
\frac{N(n)!}{N(n)^{\,C(n)}} \longrightarrow 0
\qquad \text{as } n \to \infty.
\]
\end{proposition}

\begin{proof}
By Stirling's formula,
\[
\log(N!) = N \log N - N + O(\log N).
\]
Hence
\[
\log\!\left(\frac{N!}{N^{C}}\right)
=
\log(N!) - C \log N
=
(N-C)\log N - N + O(\log N).
\]

Since $N(n) - C(n) = O(n^{k-1})$ while 
$N(n) = \Theta(n^k)$, we have
\[
(N-C)\log N = O(n^{k-1}\log n),
\]
whereas
\[
N = \Theta(n^k).
\]
Therefore,
\[
(N-C)\log N - N + O(\log N)
=
-\Theta(n^k) + O(n^{k-1}\log n)
\longrightarrow -\infty.
\]

Exponentiating both sides yields
\[
\frac{N(n)!}{N(n)^{\,C(n)}} \longrightarrow 0.
\]
\end{proof}

\begin{proposition}[Naive Underconstrained Persistence]
Let $k \ge 1$ and $\alpha > \beta \ge 0$ be fixed constants. 
Suppose the node count $N(n)$ and the naive constraint count $C(n)$ satisfy
\[
N(n) = \alpha n^k + O(n^{k-1}), 
\qquad
C(n) = \beta n^k + O(n^{k-1}).
\]

Assume, under the random-labeling heuristic, that each constraint is
satisfied with probability at least
\[
\frac{1}{\gamma N(n)}
\]
for some constant $\gamma>0$ independent of $n$. Then the expected number
of admissible labelings is bounded below by a quantity tending to infinity.

Then
\[
\frac{N(n)!}{(\gamma N(n))^{\,C(n)}} \longrightarrow \infty
\qquad \text{as } n \to \infty.
\]
\end{proposition}

\begin{proof}
By Stirling's formula,
\[
\log(N!) = N \log N - N + O(\log N).
\]
Hence
\[
\log\!\left(\frac{N!}{(\gamma N)^{C}}\right)
=
\log(N!) - C \log N  - C \log \gamma 
=
(N-C)\log N - N + O(\log N) - C \log \gamma .
\]

Since
\[
N(n)-C(n)=(\alpha-\beta)n^k+O(n^{k-1})=\Theta(n^k),
\]
we have
\[
(N-C)\log N=\Theta(n^k\log n),
\]
whereas
\[
N=\Theta(n^k),
\qquad
C\log\gamma=O(n^k).
\]
Therefore,
\[
(N-C)\log N - N - C\log\gamma + O(\log N)
=
\Theta(n^k\log n)-O(n^k)
\longrightarrow +\infty.
\]
Thus the expected number of admissible labelings tends to infinity.
\end{proof}

This comparison with other magic figures, together with the heuristic
enumeration estimates above, suggests that trihexagonal magic figures
belong to a critical regime in which naive random heuristics predict the
disappearance of solutions, while the remaining structural degrees of
freedom leave open the possibility of infinitely many admissible labelings,
as seen in other highly constrained magic figures.

Determining which of these two behaviors actually occurs requires a deeper 
understanding of the algebraic compatibility and global symmetry of the 
underlying constraints.

\section{Conjecture}

\begin{conjecture}[Eventual nonexistence under balanced constraints]
For trihexagonal magic figures (OEIS~A391236)\cite{OEISA391236}, the number of solutions
vanishes for all sufficiently large $n$.

The same phenomenon is expected to occur for the magic triangle system
with constant $2\times2$ subtriangle sums (OEIS~A375416)\cite{OEISA375416}.
\end{conjecture}

\begin{conjecture}[Persistence below the obstruction threshold]
Apart from explicit obstructions such as parity constraints,
solutions exist for all admissible values of $n$ below the threshold
at which overconstraint takes effect.
\end{conjecture}

If the second conjecture holds, then the disappearance of solutions is
expected to occur through a genuine satisfiability threshold: there are
neither sporadic missing orders below the threshold nor isolated
solvable orders after the phase transition.

Among triangularly symmetric magic-figure systems with asymptotic
constraint density 1, the trihexagonal family counted by OEIS~A391236
appears to be a more favorable testbed than the compact magic-triangle
system OEIS~A375416 for investigating eventual nonexistence.

The reason is twofold. First, relative to the order (and likewise to the
number of nodes), the observed number of solutions in A391236 is much
smaller than in A375416. Second, A391236 admits immediate congruence-based
exclusions: in particular, the currently recorded data indicate that
$a(4m)=a(4m-1)=0$, so many values of $n$ may be discarded without
computation. By contrast, the known values of A375416 grow rapidly at
small orders, making exhaustive exploration substantially more expensive.

For this reason, A391236 provides a particularly convenient experimental
setting for testing the broader hypothesis that highly symmetric
region-based magic figures with asymptotic constraint density 1
eventually cease to exist.

\section{Conclusion}

The problem studied in this paper belongs naturally to the class of
combinatorial designs arising from triangular and hexagonal tilings,
where constraints are imposed on local regions rather than along
individual edges or lines.

Beyond the specific constructions presented here, such systems provide a
useful framework for investigating asymptotically balanced constraint
regimes, in which the number of constraints grows at the same rate as the
number of variables. As shown by the naive counting argument, these
regimes are predicted to collapse; however, the examples considered in
this paper demonstrate that the actual behavior depends crucially on the
structural interaction of constraints.

In particular, families such as the trihexagonal magic figures
(OEIS~A391236) and the subtriangle-based magic triangles
(OEIS~A375416) suggest that one may study the existence of solutions as
$n$ increases, and investigate whether a phase transition from existence
to nonexistence occurs.

From this perspective, the present problem is of independent interest not
only as a specific construction, but also as a testbed for exploring
threshold phenomena in highly symmetric, region-based combinatorial
designs.

In this sense, the present construction may be viewed as a concrete model
for studying phase transition phenomena in combinatorial design problems
arising from regular tilings. By examining the existence of solutions as
$n$ increases within a fixed geometric framework, one can empirically and
theoretically investigate the boundary between persistent solvability and
eventual overconstraint.

\appendix

\section*{Appendix}
\addcontentsline{toc}{section}{Appendix}

\section{General Construction of Magic Quadrangles for Arbitrary Order \(n\) (OEIS A355256)}

A construction establishing the existence of magic quadrangles(OEIS A355256),\cite{OEISA355256,JeongQuadrangle}  has previously appeared in an implicit, example-based form in Dairyong Jeong's \textit{Susja Nori}\cite{JeongSite}, where separate constructions for odd and even orders are presented \cite{JeongOdd,JeongEven}.

As the original source is no longer readily accessible and no explicit formulation was given, we record here a precise and systematic description.

\begin{proposition}[constuction of magic quadrangle of odd order]
Let $n \ge 1$, and index the array by $0 \le i,j \le n-1$. Define
\[
i^* =
\begin{cases}
i, & \text{if } j \equiv 0 \pmod{2},\\
n-1-i, & \text{if } j \equiv 1 \pmod{2},
\end{cases}
\qquad
j^* =
\begin{cases}
j, & \text{if } i \equiv 0 \pmod{2},\\
n-1-j, & \text{if } i \equiv 1 \pmod{2}.
\end{cases}
\]
Then the array $A=(A_{i,j})$ given by
\[
A_{i,j} = n i^* + j^* + 1
\]
is exactly the array obtained from the construction described in OEIS~A355256.
\end{proposition}

\begin{proof}
The construction of OEIS~A355256 is obtained by applying reflections of the base array
\[
B_{i,j} = ni + j + 1
\]
according to the parity of $(i,j)$: entries with both indices even remain unchanged, entries with both indices odd are reflected horizontally, entries with $j$ odd are reflected vertically, and entries with $i$ odd are rotated by $180^\circ$.

The definition of $i^*$ and $j^*$ encodes precisely these four cases, and therefore produces the same array.
\end{proof}

\begin{proof}
It suffices to verify the vertex sum of each unit quadrangle.  Consider the
unit quadrangle with vertices
\[
(i,j),\quad (i,j+1),\quad (i+1,j),\quad (i+1,j+1).
\]
Since $j$ and $j+1$ have opposite parity, we have
\[
i^*(i,j)+i^*(i,j+1)=n-1
\]
and
\[
i^*(i+1,j)+i^*(i+1,j+1)=n-1.
\]
Similarly, since $i$ and $i+1$ have opposite parity,
\[
j^*(i,j)+j^*(i+1,j)=n-1
\]
and
\[
j^*(i,j+1)+j^*(i+1,j+1)=n-1.
\]
Therefore
\[
\begin{aligned}
&A_{i,j}+A_{i,j+1}+A_{i+1,j}+A_{i+1,j+1}\\
&=n\cdot 2(n-1)+2(n-1)+4\\
&=2n^2+2.
\end{aligned}
\]
Hence every unit quadrangle has the same vertex sum.
\end{proof}

\begin{proposition}[Explicit form of the even-order example]
Let $k\ge 1$, and index the entries of a $2k\times 2k$ array by
\[
0\le i,j\le 2k-1.
\]
Write
\[
a=\left\lfloor \frac{i}{2}\right\rfloor,\qquad
b=\left\lfloor \frac{j}{2}\right\rfloor .
\]
Then the array $A=(A_{i,j})$ is given by
\[
A_{i,j}=
\begin{cases}
kb+a+1, & i\equiv 0\pmod 2,\ j\equiv 0\pmod 2,\\[4pt]
3k^2-(kb+a), & i\equiv 0\pmod 2,\ j\equiv 1\pmod 2,\\[4pt]
4k^2-(kb+a), & i\equiv 1\pmod 2,\ j\equiv 0\pmod 2,\\[4pt]
k^2+kb+a+1, & i\equiv 1\pmod 2,\ j\equiv 1\pmod 2.
\end{cases}
\]
In particular, the four parity classes occupy the four disjoint intervals
\[
\{1,\dots,k^2\},\quad
\{k^2+1,\dots,2k^2\},\quad
\{2k^2+1,\dots,3k^2\},\quad
\{3k^2+1,\dots,4k^2\}.
\]
\end{proposition}

\begin{proof}

For any \(2\times2\) subsquare, regardless of its position in the
array, label its four entries according to the parities of their row
and column indices modulo \(2\).  More precisely, if an entry lies in
row \(i\) and column \(j\), we assign it the type
\[
(\overline{i},\overline{j})
\in\{00,01,10,11\},
\]
where
\[
\overline{i}\equiv i\pmod 2,
\qquad
\overline{j}\equiv j\pmod 2.
\]
Thus every \(2\times2\) subsquare contains exactly one entry of each
type \(00,01,10,11\).

For each such entry, write
\[
a=\left\lfloor \frac{i}{2}\right\rfloor,
\qquad
b=\left\lfloor \frac{j}{2}\right\rfloor,
\qquad
t=kb+a.
\]
According to the parity type, we denote these values by
\[
t_{00},\quad t_{01},\quad t_{10},\quad t_{11}.
\]
That is, \(t_{rs}\) denotes the quantity \(kb+a\) corresponding to the
unique entry of parity type \(rs\) in the given \(2\times2\)
subsquare.

With this notation, the sum of the four entries in any \(2\times2\)
subsquare can be written as
\[
(t_{00}+1)+(3k^2-t_{01})+(4k^2-t_{10})+(k^2+t_{11}+1).
\]
Collecting the constant and variable parts, this becomes
\[
8k^2+2+\bigl(t_{00}+t_{11}-(t_{01}+t_{10})\bigr).
\]
Thus it suffices to prove that
\[
t_{00}+t_{11}=t_{01}+t_{10}.
\]

If the two entries with \(i\equiv 0\pmod 2\) have their \(a\)-values
larger by \(1\), or similarly if the two entries with
\(j\equiv 0\pmod 2\) have their \(b\)-values larger by \(1\), then the
corresponding change appears once on each side of
\[
t_{00}+t_{11}=t_{01}+t_{10},
\]
so the equality still holds.

The same remains true when both phenomena occur simultaneously.  In
each case, a direct calculation shows that
\[
t_{00}+t_{11}=t_{01}+t_{10}.
\]

\end{proof}

We additionally note that the phenomenon observed for the trihexagonal
magic figure—namely, that the magic-sum invariance extends from unit
polygons to larger blocks of the same shape—also appears in the magic
quadrangle.

\begin{proposition}
In the magic quadrangle \(\mathrm{A355256}\),\cite{OEISA355256}
the corner sum over any \((2r+1)\times(2s+1)\) rectangular block is constant.
Here the size is counted by side length, so the smallest rectangular
block is regarded as \(1\times1\).
\end{proposition}

\begin{proof}
Let \(M\) denote the common sum of every \(2\times2\) subsquare.

First consider the case where one side has length \(1\).  Suppose that
two adjacent entries in a row have sum \(K\).  Since each \(2\times2\)
subsquare has sum \(M\), the two adjacent entries immediately above or
below them have sum \(M-K\).  Moving one step further in the same
direction gives sum \(K\) again.  Hence these adjacent two-entry sums
alternate between \(K\) and \(M-K\).

It follows that the sum of any \(1\times(2r+1)\) or
\((2r+1)\times1\) block is independent of its position.  

Now consider a general \((2r+1)\times(2s+1)\) rectangular block.
Decompose it into \(2s+1\) parallel strips of size
\((2r+1)\times1\). Applying the same alternating principle, the sums
of the two corner entries on the longer side of these strips alternate
in the same way.

Since the number of strips is odd, the total corner sum of the whole
rectangle is equal to that of the original \((2r+1)\times1\) strip.
Thus every \((2r+1)\times(2s+1)\) rectangular block has the same sum.
\end{proof}

\section{Computational verification of Trihexagonal magic figure for n=2 and 5}

The complete computational materials for Trihexagonal magic figure\cite{OEISA391236} are publicly available through the
author’s GitHub repository

\url{https://github.com/gwahak/mathematics}

Links to the specific notebook files used for the enumeration are also
provided in the OEIS entry for the corresponding sequence
{OEIS A391236}\cite{OEISA391236}.

\subsection{Computational verification for order \(n=2\)}

For the smallest nontrivial case \(n=2\), the existence and exact number
of solutions can be verified by a direct exhaustive search.

We denote the nine vertices by
\[
\{a_{00},a_{01},a_{02},a_{03},a_{11},a_{13},a_{22},a_{23},a_{33}\},
\]
arranged as
\[
\begin{array}{ccccccc}
&&& a_{00} \\
&& a_{01} && a_{11} \\
& a_{02} &&&& a_{22} \\
a_{03} && a_{13} && a_{23} && a_{33}
\end{array}
\]

and assign the numbers \(1,2,\dots,9\) bijectively.

The magic conditions require that all unit triangles and the outer
triangle have the same sum:
\[
a_{22}+a_{23}+a_{33}=15,
\]
\[
a_{00}+a_{11}+a_{01}=15,
\]
\[
a_{02}+a_{03}+a_{13}=15,
\]
\[
a_{00}+a_{33}+a_{03}=15.
\]

To avoid counting symmetric duplicates arising from reflection, we impose
the canonical ordering
\[
a_{33} > a_{00} > a_{03}.
\]

The following short Python program performs the complete enumeration.

\begin{verbatim}
import itertools

count = 0

for a00, a01, a02, a03, a11, a13, a22, a23, a33 
                    in itertools.permutations(range(1, 10), 9):

    if a22 + a23 + a33 != 15:
        continue
    if a00 + a11 + a01 != 15:
        continue
    if a02 + a03 + a13 != 15:
        continue
    if a00 + a33 + a03 != 15:
        continue

    if not (a33 > a00 > a03):
        continue

    count += 1

print(count)
\end{verbatim}

The program returns \textbf{for order 2, number of solutions = 80}

\subsection{Computational verification for order \(n=5\)}

\lstset{
  basicstyle=\ttfamily\footnotesize,
  breaklines=true,
  breakatwhitespace=false,
  columns=fullflexible,
  keepspaces=true,
  frame=single,
  numbers=left,
  numberstyle=\tiny,
  xleftmargin=1em
}
\begin{lstlisting}[language=Python,caption={Enumeration code for order \(n=5\).}]

i=0
for a57 in range(24,46):
  for a24 in range(int((69-a57)/2)+1, min(a57,69-a57)):
    a27 = 69 - a57 - a24
    for a45 in set(range(1,46))-{a24,a27,a57}:
      for a48 in set(range(1,46))-{a24,a27,a57,a45}:
        a15=69-a45-a48
        if a15 in set(range(1,46))-{a24,a27,a57,a45,a48}:
          for a46 in set(range(1,46))-{a24,a27,a57,a45,a48,a15}:
            a47=69-a46-a57
            a37=69-a47-a48
            a26=69-a37-a27
            a25=69-a26-a15
            a35=69-a25-a24
            if {a46,a47,a37,a26,a25,a35}.issubset(set(range(1,46))-{a24,a27,a57,a45,a48,a15}):
              if len({a46,a47,a37,a26,a25,a35})==6:
                set33=set(range(1,46))-{a24,a27,a57,a45,a48,a15,a46,a47,a37,a26,a25,a35}
                for a67 in set33:
                  a68=69-a67-a57
                  if a67!=a68 and a68 in set33:
                    for a66 in set33-{a67,a68}:
                      a77=69-a66-a67
                      a55=138-a66-a67-a57-a46-a45
                      a44=69-a45-a55
                      if {a66,a77,a55,a44}.issubset(set33-{a67,a68}):
                        if len({a66,a77,a55,a44})==4:
                          for a79 in set33-{a67,a68,a66,a77,a55,a44}:
                            a99=a77+a67+a68+a79-69
                            a69=69-a79-a68
                            a59=138-a69-a68-a57-a47-a48
                            a49=69-a48-a59
                            if {a79,a99,a69,a59,a49}.issubset(set33-{a67,a68,a66,a77,a55,a44}):
                              if len({a79,a99,a69,a59,a49})==5:
                                for a28 in set33-{a67,a68,a66,a77,a55,a44}-{a79,a99,a69,a59,a49}:
                                  a39=138-a28-a27-a37-a48-a49
                                  a29=69-a39-a28
                                  a17=69-a27-a28
                                  if {a28,a39,a29,a17}.issubset(set33-{a67,a68,a66,a77,a55,a44}-{a79,a99,a69,a59,a49}):
                                    if len({a28,a39,a29,a17})==4:
                                      for a23 in set33-{a67,a68,a66,a77,a55,a44}-{a79,a99,a69,a59,a49}-{a28,a39,a29,a17}:
                                        a33=138-a23-a24-a35-a45-a44
                                        a22=69-a23-a33
                                        a13=69-a23-a24
                                        if {a23,a33,a22,a13}.issubset(set33-{a67,a68,a66,a77,a55,a44}-{a79,a99,a69,a59,a49}-{a28,a39,a29,a17}):
                                          if len({a23,a33,a22,a13})==4:
                                            for a04 in set33-{a67,a68,a66,a77,a55,a44}-{a79,a99,a69,a59,a49}-{a28,a39,a29,a17}-{a23,a33,a22,a13}:
                                              a05=69-a15-a04
                                              a03=138-a04-a15-a25-a24-a13
                                              a02=69-a13-a03
                                              a00=a22+a23+a13+a02-69
                                              a06=138-a05-a15-a26-a27-a17
                                              a07=69-a06-a17
                                              a09=a07+a17+a28+a29-69
                                              if {a00,a02,a03,a04,a05,a06,a07,a09}.issubset(set33-{a67,a68,a66,a77,a55,a44}-{a79,a99,a69,a59,a49}-{a28,a39,a29,a17}-{a23,a33,a22,a13}):
                                                if len({a00,a02,a03,a04,a05,a06,a07,a09})==8:
                                                  set6=set33-{a67,a68,a66,a77,a55,a44}-{a79,a99,a69,a59,a49}-{a28,a39,a29,a17}-{a23,a33,a22,a13}-{a00,a02,a03,a04,a05,a06,a07,a09}
                                                  for a88 in set6:
                                                    a89=69-a88-a99
                                                    if a89 in set6-{a88} and a88>a89:
                                                      for a08 in set6-{a88,a89}:
                                                        a19= 69-a08-a09
                                                        if a19 in set6-{a88,a89,a08} and a08>a19:
                                                          for a01 in set6-{a88,a89,a08,a19}:
                                                            a11=69-a01-a00
                                                            if a11 in set6-{a88,a89,a08,a19,a01} and a01>a11:
                                                              i+=1
                                                              print(i,"th: solution:", a99,a88,a89,a77,a79,a66,a67,a68,a69,a55,a57,a59,a44,a45,a46,a47,a48,a49,a33,a35,a37,a39,a22,a23,a24,a25,a26,a27,a28,a29,a11,a13,a15,a17,a19,a00,a01,a02,a03,a04,a05,a06,a07,a08,a09)
\end{lstlisting}

\end{document}